\documentclass{amsart}
\usepackage{eurosym}
\usepackage{amsfonts}
\usepackage{amsmath}

\newtheorem{theorem}{Theorem}[section]

\newtheorem{lemma}[theorem]{Lemma}
\theoremstyle{remark}

\theoremstyle{definition}

\theoremstyle{proposition}
\newtheorem{proposition}[theorem]{Proposition}
\numberwithin{equation}{section}
 
\begin{document}
\title{Conical structure for shrinking Ricci solitons}
\author{Ovidiu Munteanu}
\email{ovidiu.munteanu@uconn.edu}
\address{Department of Mathematics, University of Connecticut, Storrs, CT
06268, USA}
\author{Jiaping Wang}
\email{jiaping@math.umn.edu}
\address{School of Mathematics, University of Minnesota, Minneapolis, MN
55455, USA}
\date{}

\begin{abstract}
For a shrinking Ricci soliton with Ricci curvature convergent to zero at
infinity, it is shown that it must be asymptotically conical.
\end{abstract}

\maketitle

\section{Introduction}

The purpose of this paper is to show that a shrinking gradient
Ricci soliton must be smoothly asymptotic to a cone if its Ricci curvature 
goes to zero at infinity. Recall that a
gradient shrinking Ricci soliton is a Riemannian manifold $\left(
M^{n},g\right) $ for which there exists a potential function $f$ such that 
\begin{equation}
\mathrm{Ric}+\mathrm{Hess}\left( f\right) =\frac{1}{2}g,  \label{soliton_eq}
\end{equation}%
where $\mathrm{Ric}$ is the Ricci curvature of $M$ and $\mathrm{Hess}\left(
f\right) $ the Hessian of $f.$ Aside from its own interest as generalization
of Einstein manifolds, Ricci solitons are important in the study of the
Ricci flows. Indeed, one easily verifies (see \cite{CLN}) that $%
g(t)=(1-t)\,\phi _{t}^{\ast }\,g,$ for $-\infty <t<1,$ is a solution to the
Ricci flow

\begin{eqnarray*}
\frac{\partial }{\partial t}g(t) &=&-2\,\mathrm{Ric}(g(t)) \\
g(0) &=&g
\end{eqnarray*}%
for a suitably chosen family of diffeomorphisms $\phi _{t}$ on $M$ with $%
\phi _{0}=id.$ So shrinking Ricci solitons may be regarded as self-similar
solutions to the Ricci flows. It has been shown in \cite{EMT} that the
blow-ups around a type-I singularity point always converge to nontrivial
gradient shrinking Ricci solitons. Therefore, it would be very desirable to
understand and even classify shrinking Ricci solitons.

In the case dimension $n=2,$ according to \cite{H}, the only examples are
either the sphere $\mathbb{S}^2$ or the Gaussian soliton, Euclidean space $%
\mathbb{R}^2$ together with potential function $f(x)=\frac{1}{4}\,|x|^2.$
For dimension $n=3, $ improving upon the breakthrough of Perelman \cite{P},
Naber \cite{N}, Ni and Wallach \cite{NW}, and Cao, Chen and Zhu \cite{CCZ}
have concluded that a three dimensional shrinking gradient Ricci soliton
must be a quotient of the sphere $\mathbb{S}^3,$ or $\mathbb{R}^3,$ or $%
\mathbb{S}^2\times \mathbb{R}.$

For high dimensional shrinking Ricci solitons, examples other than the
sphere and Gaussian soliton (and their products) have been constructed by 
\cite{Ca, K, WZ, FIK, DW}. This certainly indicates that it would be more
complicated, if at all possible, to obtain a complete classification. Under
some auxiliary conditions on the full curvature tensor, partial
classification results have been established. In \cite{N}, Naber has shown
that a four dimensional complete shrinking Ricci soliton of bounded
nonnegative curvature operator must be a quotient of $\mathbb{R}^{k}\times 
\mathbb{S}^{4-k}$ with $k=0,1,2.$ A theorem of B\"{o}hm and Wilking \cite{BW}
implies that a compact shrinking Ricci soliton of any dimension with
positive curvature operator must be a spherical space form. Also, shrinking
gradient Ricci solitons of vanishing Weyl tensor have been classified and must
be finite quotients of sphere $\mathbb{S}^{n},$ or $\mathbb{R}^n,$ or
$\mathbb{S}^{n-1}\times \mathbb{R}$ (see \cite{Z, ELM, PW, CWZ, MS}). More generally,
in a recent work \cite{CC}, Cao and Chen have shown that a Bach-flat gradient shrinking
Ricci soliton is either Einstein, or a finite quotient of the Gaussian shrinking soliton $\mathbb{R}^n,$
or a finite quotient of $N^{n-1}\times \mathbb{R},$ where $N^{n-1}$ is an Einstein manifold of positive
scalar curvature. We refer the reader to the two surveys \cite{Ca1,
Ca2} for more results and details.

In another direction, Kotschwar and Wang \cite{KW} have recently shown that
two shrinking Ricci solitons $C^2$ asymptotic to the same cone must be
isometric. Here, by a cone, we mean a manifold $[0,\infty )\times \Sigma $
endowed with Riemannian metric $g_{c}=dr^{2}+r^{2}\,g_{\Sigma },$ where $%
(\Sigma ,g_{\Sigma })$ is a closed $(n-1)$-dimensional Riemannian manifold.
Denote $E_{R}=(R,\infty )\times \Sigma $ for $R\geq 0$ and define the
dilation by $\lambda $ to be the map $\rho _{\lambda }:E_{0}\rightarrow
E_{0} $ given by $\rho _{\lambda }(r,\sigma )=(\lambda \,r,\sigma ).$ A
Riemannian manifold $(M,g)$ is said to be $C^{k}$ asymptotic to the cone $%
(E_{0},g_{c})$ if, for some $R>0,$ there is a diffeomorphism $\Phi
:E_{R}\rightarrow M\setminus \Omega $ such that $\lambda ^{-2}\,\rho
_{\lambda }^{\ast }\,\Phi ^{\ast }\,g\rightarrow g_{c}$ as $\lambda
\rightarrow \infty $ in $C_{loc}^{k}(E_{0},g_{c}),$ where $\Omega $ is a
compact subset of $M.$

In view of their result, it becomes an interesting question to determine
when a shrinking Ricci soliton is asymptotically conical. In our recent work 
\cite{MW}, we have shown that this is the case for four dimensional
shrinking gradient Ricci solitons with scalar curvature converging to $0$ at
infinity. This result depends on the fact that the full curvature tensor 
$\mathrm{Rm}$ of a four dimensional soliton is controlled by its scalar curvature $S$
alone, that is, $\left\vert \mathrm{Rm}\right\vert \leq c\,\left\vert 
\mathrm{Ric}\right\vert \leq c\,S.$ While it remains to be seen whether such an
estimate is true for high dimensional shrinking Ricci solitons, by imposing
assumption on the Ricci curvature instead, we manage to obtain a parallel
result as well.

\begin{theorem}
\label{decay}Let $\left( M,g,f\right) $ be a gradient shrinking Ricci
soliton of dimension $n$ with Ricci curvature convergent to zero at
infinity. Then $\left( M,g,f\right) $ is $C^k$ asymptotic to a cone for all 
$k.$
\end{theorem}

Consequently, the classification problem for such solitons is reduced to the
one for cones.

Essential to the proof of Theorem \ref{decay} is a quadratic decay estimate
for the Riemann curvature $\left\vert \mathrm{Rm}\right\vert .$ Once this is
available, together with Shi's \cite{S} derivative estimates of $\mathrm{Rm},$ 
it is then straightforward to conclude that $\left( M,g\right) $ is
asymptotically conical \cite{KW}. As demonstrated in \cite{MW}, such a decay
estimate follows from a maximum principle argument provided that the Riemann
curvature tensor $\mathrm{Rm}$ converges to zero at infinity. So the heart
of the proof is to conclude from $\mathrm{Ric}$ converging to $0$ that $%
\mathrm{Rm}$ goes to $0$ as well. Here, we are very much inspired by the
work of \cite{MWa}, where it says that for a shrinking Ricci soliton, its
Riemann curvature is at most of polynomial growth if its Ricci curvature is
bounded. However, we would like to point out that our argument differs
significantly from \cite{MWa} in terms of technical details.

We remark that our argument only requires the Ricci curvature being sufficiently
small outside a compact set. More precisely, Theorem \ref{decay} continues to hold
if one assumes instead that $\vert \mathrm{Ric}\vert \leq \delta$ near the infinity of $M$
for some positive constant $\delta$ depending only on the dimension $n.$ 

\section{Curvature estimates}

In this section, we prove Theorem \ref{decay}. We continue to denote by $%
(M,g,f)$ an $n$-dimensional shrinking Ricci soliton with potential function $%
f.$

Let us recall the following important identities 
\begin{eqnarray}
\nabla _{k}R_{jk} &=&R_{jk}f_{k}=\frac{1}{2}\nabla _{j}S  \label{id} \\
\nabla _{l}R_{ijkl} &=&R_{ijkl}f_{l}=\nabla _{j}R_{ki}-\nabla _{i}R_{kj}. 
\notag
\end{eqnarray}%
As observed in \cite{H}, this implies $S+|\nabla f|^{2}=f$ by adding a
suitable constant to $f.$ Since $S\geq 0$ by \cite{C}, we have $|\nabla
f|^{2}\leq f.$

Also, denoting $\Delta _{f}=\Delta -\left\langle \nabla f,\nabla
\right\rangle ,$ we have 

\begin{eqnarray}
\Delta _{f}R_{ij} &=&R_{ij}-2R_{ikjl}R_{kl}  \label{id2} \\
\Delta _{f}\mathrm{Rm} &=&\mathrm{Rm}+\mathrm{Rm}\ast \mathrm{Rm}  \notag
\end{eqnarray}

Let us denote 
\begin{equation*}
D\left( r\right) :=\left\{ x\in M:\;f\left( x\right) \leq r\right\} .
\end{equation*}%
Notice that $D\left( r\right) $ is always compact as by \cite{CZ} there
exists constant $c$ such that

\begin{equation}
\frac{1}{4}r^{2}\left( x\right) -c\,r(x)\leq f\left( x\right) \leq \frac{1}{4%
}r^{2}\left( x\right) +c\,r(x)\;\;\text{for}\;\;r\left( x\right) \geq 1.
\label{f}
\end{equation}%
Here $r\left( x\right) $ is the distance from $x$ to a fixed point $x_{0}\in
M.$ Also, recall ( see \cite{CZ}) that the volume $V(r)$ of $D(r)$ satisfies 
\begin{equation}
V(r)\leq c\,r^{\frac{n}{2}}.  \label{V}
\end{equation}

We define the cut-off $\phi $ with support in $D\left( r\right) $ by

\begin{equation*}
\phi \left( x\right) =\left\{ 
\begin{array}{ccc}
\frac{1}{r}\left( r-f\left( x\right) \right) & \text{if} & x\in D\left(
r\right) \\ 
0 & \text{if} & x\in M\backslash D\left( r\right)%
\end{array}%
\right.
\end{equation*}%
Let us choose $r_{0}>0$ large enough so that  $f\geq 1$ and 
\begin{equation}
\left\vert \mathrm{Ric}\right\vert \leq \frac{1}{p^{5}} \text{ \ on }%
M\backslash D\left( r_{0}\right) .  \label{a2}
\end{equation}%
In particular, since 
\begin{equation*}
\mathrm{Ric}+\mathrm{Hess}\left( f\right) =\frac{1}{2}g,
\end{equation*}%
we have

\begin{equation}
\mathrm{Hess}(f)\geq \frac{1}{3}\,g \text{ \ on } M\backslash D\left( r_{0}\right). \label{a2s}
\end{equation}%
We fix $p\geq 8\,n$ large enough depending only on dimension $n$ and let $q$
and $a$ be constants satisfying

\begin{equation}
q\geq 2p+3\text{ and }a\leq \frac{1}{4}p.  \label{const}
\end{equation}%
Throughout the paper, unless otherwise indicated, we will use $C$ to denote constants that may depend on
the geometry of $D\left( r_{0}\right) ,$ $c$ constants depending only on
dimension $n$ but independent of $p$, and $c\left( p\right) $ constants
depending on $p$. These constants may change from line to line. We first
prove the following lemma.

\begin{lemma}
\label{int} Let $\left( M,g,f\right) $ be a gradient shrinking Ricci soliton
of dimension $n$ with $\lim_{x\to \infty} \left\vert \mathrm{Ric}\right\vert (x)=0.$ 
Then there exist positive constants $c$, $C$ and $c\left( p\right) $
such that for $\alpha \in \left\{ 0,1\right\} ,$

\begin{eqnarray*}
\int_{M}\left\vert \mathrm{Rm}\right\vert ^{p}f^{a}\phi ^{q} &\leq &\frac{c}{%
p}\int_{M}\left\vert \mathrm{Rm}\right\vert ^{p-1+\alpha }f^{a}\phi ^{q}+%
\frac{c}{p}\int_{M}\left\vert \mathrm{Rm}\right\vert ^{p+\alpha }f^{a}\phi
^{q} \\
&&+c\left( p\right) \int_{M}\left\vert \mathrm{Ric}\right\vert ^{p}f^{a}\phi
^{q}+C.
\end{eqnarray*}
\end{lemma}

\begin{proof}
Integrating by parts and using that $\Delta f\leq \frac{n}{2},$ we get%
\begin{eqnarray*}
-\frac{n}{2}\int_{M}\left\vert \mathrm{Rm}\right\vert ^{p}f^{a}\phi ^{q}
&\leq &-\int_{M}\left\vert \mathrm{Rm}\right\vert ^{p}\left( \Delta f\right)
f^{a}\phi ^{q} \\
&=&\int_{M}\left\langle \nabla \left\vert \mathrm{Rm}\right\vert ^{p},\nabla
f\right\rangle f^{a}\phi ^{q} \\
&&+a\int_{M}\left\vert \mathrm{Rm}\right\vert ^{p}\left\vert \nabla
f\right\vert ^{2}f^{a-1}\phi ^{q} \\
&&+\int_{M}\left\vert \mathrm{Rm}\right\vert ^{p}f^{a}\left\langle \nabla
f,\nabla \phi ^{q}\right\rangle  \\
&\leq &\int_{M}\left\langle \nabla \left\vert \mathrm{Rm}\right\vert
^{p},\nabla f\right\rangle f^{a}\phi ^{q}+a\int_{M}\left\vert \mathrm{Rm}%
\right\vert ^{p}f^{a}\phi ^{q},
\end{eqnarray*}%
where in the last line we have used that $\left\langle \nabla f,\nabla \phi
^{q}\right\rangle \leq 0$. Therefore, by Bianchi identities we obtain that 
\begin{eqnarray*}
-\left( a+\frac{n}{2}\right) \int_{M}\left\vert \mathrm{Rm}\right\vert
^{p}f^{a}\phi ^{q} &\leq &\int_{M}\left\langle \nabla \left\vert \mathrm{Rm}%
\right\vert ^{p},\nabla f\right\rangle \;f^{a}\phi ^{q} \\
&=&p\int_{M}f_{h}\left( \nabla _{h}R_{ijkl}\right) R_{ijkl}\left\vert 
\mathrm{Rm}\right\vert ^{p-2}f^{a}\phi ^{q} \\
&=&2p\int_{M}f_{h}\left( \nabla _{l}R_{ijkh}\right) R_{ijkl}\left\vert 
\mathrm{Rm}\right\vert ^{p-2}f^{a}\phi ^{q}.
\end{eqnarray*}%
It follows through integration by parts that 
\begin{eqnarray}
&&-\left( a+\frac{n}{2}\right) \int_{M}\left\vert \mathrm{Rm}\right\vert
^{p}f^{a}\phi ^{q}  \label{a1} \\
&\leq &-2p\int_{M}R_{ijkh}f_{hl}R_{ijkl}\left\vert \mathrm{Rm}\right\vert
^{p-2}f^{a}\phi ^{q}  \notag \\
&&-2p\int_{M}R_{ijkh}f_{h}\left( \nabla _{l}R_{ijkl}\right) \left\vert 
\mathrm{Rm}\right\vert ^{p-2}f^{a}\phi ^{q}  \notag \\
&&-2p\int_{M}R_{ijkh}f_{h}R_{ijkl}\left( \nabla _{l}\left\vert \mathrm{Rm}%
\right\vert ^{p-2}\right) f^{a}\phi ^{q}  \notag \\
&&-2ap\int_{M}R_{ijkh}f_{h}R_{ijkl}f_{l}\left\vert \mathrm{Rm}\right\vert
^{p-2}f^{a-1}\phi ^{q}  \notag \\
&&+\frac{2pq}{r}\int_{M}R_{ijkh}f_{h}R_{ijkl}f_{l}\left\vert \mathrm{Rm}%
\right\vert ^{p-2}f^{a}\phi ^{q-1}.  \notag
\end{eqnarray}%
Note that on $M\backslash D\left( r_{0}\right),$ by (\ref{a2s}),
 
\begin{equation*}
-R_{ijkh}f_{hl}R_{ijkl}\leq -\frac{1}{3}\left\vert \mathrm{Rm}\right\vert
^{2}.
\end{equation*}%
Together with (\ref{id}),  it results from (\ref{a1}) that 

\begin{eqnarray}
\frac{p}{3}\int_{M}\left\vert \mathrm{Rm}\right\vert ^{p}f^{a}\phi ^{q}
&\leq &-2p\int_{M}R_{ijkh}f_{h}R_{ijkl}\left( \nabla _{l}\left\vert \mathrm{%
Rm}\right\vert ^{p-2}\right) f^{a}\phi ^{q}  \label{a3} \\
&&+\frac{2pq}{r}\int_{M}\left\vert R_{ijkh}f_{h}\right\vert ^{2}\left\vert 
\mathrm{Rm}\right\vert ^{p-2}f^{a}\phi ^{q-1}  \notag \\
&&+C.  \notag
\end{eqnarray}%
By (\ref{id}) again, we have 

\begin{eqnarray*}
&&-2p\int_{M}R_{ijkh}f_{h}R_{ijkl}\left( \nabla _{l}\left\vert \mathrm{Rm}%
\right\vert ^{p-2}\right) f^{a}\phi ^{q} \\
&\leq &cp^{2}\int_{M}\left\vert \nabla \mathrm{Ric}\right\vert \left\vert
\nabla \mathrm{Rm}\right\vert \left\vert \mathrm{Rm}\right\vert
^{p-2}f^{a}\phi ^{q}.
\end{eqnarray*}
For $\alpha \in \left\{ 0,1\right\} ,$ it follows from (\ref{a3}) that 

\begin{eqnarray}
\int_{M}\left\vert \mathrm{Rm}\right\vert ^{p}f^{a}\phi ^{q} &\leq
&cp^{3}\int_{M}\left\vert \nabla \mathrm{Ric}\right\vert ^{2}\left\vert 
\mathrm{Rm}\right\vert ^{p-1-\alpha }f^{a}\phi ^{q}  \label{a5} \\
&&+\frac{c}{p}\int_{M}\left\vert \nabla \mathrm{Rm}\right\vert
^{2}\left\vert \mathrm{Rm}\right\vert ^{p-3+\alpha }f^{a}\phi ^{q}.  \notag
\\
&&+\frac{cp}{r}\int_{M}\left\vert R_{ijkh}f_{h}\right\vert ^{2}\left\vert 
\mathrm{Rm}\right\vert ^{p-2}f^{a}\phi ^{q-1}+C.  \notag
\end{eqnarray}

We now estimate the first term on the right side of (\ref{a5}). 
Note that (\ref{id2}) implies 
\begin{eqnarray*}
\Delta _{f}\left( \left\vert \mathrm{Ric}\right\vert ^{2}\left\vert \mathrm{%
Rm}\right\vert ^{p-1-\alpha }\right)  &=&\left( \Delta _{f}\left\vert 
\mathrm{Ric}\right\vert ^{2}\right) \left\vert \mathrm{Rm}\right\vert
^{p-1-\alpha }+\left\vert \mathrm{Ric}\right\vert ^{2}\Delta _{f}\left\vert 
\mathrm{Rm}\right\vert ^{p-1-\alpha } \\
&&+2\left\langle \nabla \left\vert \mathrm{Ric}\right\vert ^{2},\nabla
\left\vert \mathrm{Rm}\right\vert ^{p-1-\alpha }\right\rangle  \\
&\geq &2\left\vert \nabla \mathrm{Ric}\right\vert ^{2}\left\vert \mathrm{Rm}%
\right\vert ^{p-1-\alpha }-cp\left\vert \mathrm{Ric}\right\vert
^{2}\left\vert \mathrm{Rm}\right\vert ^{p-\alpha } \\
&&-cp\left\vert \nabla \mathrm{Ric}\right\vert \left\vert \nabla \mathrm{Rm}%
\right\vert \left\vert \mathrm{Ric}\right\vert \left\vert \mathrm{Rm}%
\right\vert ^{p-2-\alpha }.
\end{eqnarray*}%
Consequently, we get 

\begin{eqnarray}
&&2\int_{M}\left\vert \nabla \mathrm{Ric}\right\vert ^{2}\left\vert \mathrm{%
Rm}\right\vert ^{p-1-\alpha }f^{a}\phi ^{q}  \label{a6} \\
&\leq &\int_{M}\Delta _{f}\left( \left\vert \mathrm{Ric}\right\vert
^{2}\left\vert \mathrm{Rm}\right\vert ^{p-1-\alpha }\right) f^{a}\phi
^{q}+cp\int_{M}\left\vert \mathrm{Ric}\right\vert ^{2}\left\vert \mathrm{Rm}%
\right\vert ^{p-\alpha }f^{a}\phi ^{q}  \notag \\
&&+cp\int_{M}\left\vert \nabla \mathrm{Ric}\right\vert \left\vert \nabla 
\mathrm{Rm}\right\vert \left\vert \mathrm{Ric}\right\vert \left\vert \mathrm{%
Rm}\right\vert ^{p-2-\alpha }f^{a}\phi ^{q}.  \notag
\end{eqnarray}

The last term in (\ref{a6}) can be estimated by 

\begin{eqnarray}
&&2\int_{M}\left\vert \nabla \mathrm{Ric}\right\vert \left\vert \nabla 
\mathrm{Rm}\right\vert \left\vert \mathrm{Ric}\right\vert \left\vert \mathrm{%
Rm}\right\vert ^{p-2-\alpha }f^{a}\phi ^{q}  \label{a8} \\
&\leq &\frac{1}{p^{5}}\int_{M}\left\vert \nabla \mathrm{Rm}\right\vert
^{2}\left\vert \mathrm{Rm}\right\vert ^{p-3+\alpha }f^{a}\phi ^{q}  \notag \\
&&+p^{5}\ \int_{M}\left\vert \nabla \mathrm{Ric}\right\vert ^{2}\left\vert 
\mathrm{Ric}\right\vert ^{2}\left\vert \mathrm{Rm}\right\vert ^{p-1-3\alpha
}f^{a}\phi ^{q}.  \notag
\end{eqnarray}%
We claim that 

\begin{eqnarray}
&&p^{5}\, \int_{M}\left\vert \nabla \mathrm{Ric}\right\vert ^{2}\left\vert 
\mathrm{Ric}\right\vert ^{2}\left\vert \mathrm{Rm}\right\vert ^{p-1-3\alpha
}f^{a}\phi ^{q}  \label{a9} \\
&\leq &\frac{1}{p^{2}}\int_{M}\left\vert \nabla \mathrm{Ric}\right\vert
^{2}\left\vert \mathrm{Rm}\right\vert ^{p-1-\alpha }f^{a}\phi ^{q}  \notag \\
&&+c\left( p\right) \alpha \int_{M}\left\vert \nabla \mathrm{Ric}\right\vert
^{2}\left\vert \mathrm{Ric}\right\vert ^{p-2}f^{a}\phi ^{q}+C.  \notag
\end{eqnarray}%
Indeed, for $\alpha =0$ this is obvious by (\ref{a2}), whereas for 
$\alpha =1,$ it follows immediately from Young's inequality. Plugging 
(\ref{a9}) into (\ref{a8}), we get 
 
\begin{eqnarray}
&&2\int_{M}\left\vert \nabla \mathrm{Ric}\right\vert \left\vert \nabla 
\mathrm{Rm}\right\vert \left\vert \mathrm{Ric}\right\vert \left\vert \mathrm{%
Rm}\right\vert ^{p-2-\alpha }f^{a}\phi ^{q}  \label{a10} \\
&\leq &\frac{1}{p^{5}}\int_{M}\left\vert \nabla \mathrm{Rm}\right\vert
^{2}\left\vert \mathrm{Rm}\right\vert ^{p-3+\alpha }f^{a}\phi ^{q}  \notag \\
&&+\frac{1}{p^{2}}\int_{M}\left\vert \nabla \mathrm{Ric}\right\vert
^{2}\left\vert \mathrm{Rm}\right\vert ^{p-1-\alpha }f^{a}\phi ^{q}  \notag \\
&&+c\left( p\right) \alpha \int_{M}\left\vert \nabla \mathrm{Ric}\right\vert
^{2}\left\vert \mathrm{Ric}\right\vert ^{p-2}f^{a}\phi ^{q}+C.  \notag
\end{eqnarray}
We now estimate
 
\begin{eqnarray}
\int_{M}\left\vert \nabla \mathrm{Ric}\right\vert ^{2}\left\vert \mathrm{Ric}%
\right\vert ^{p-2}f^{a}\phi ^{q} &\leq &\int_{M}\left( \Delta \left\vert 
\mathrm{Ric}\right\vert ^{2}\right) \left\vert \mathrm{Ric}\right\vert
^{p-2}f^{a}\phi ^{q}  \label{a11} \\
&&-\int_{M}\left\langle \nabla f,\nabla \left\vert \mathrm{Ric}\right\vert
^{2}\right\rangle \left\vert \mathrm{Ric}\right\vert ^{p-2}f^{a}\phi ^{q} 
\notag \\
&&+c\int_{M}\left\vert \mathrm{Rm}\right\vert \left\vert \mathrm{Ric}%
\right\vert ^{p}f^{a}\phi ^{q}  \notag \\
&\leq &-a\int_{M}\left\langle \nabla f,\nabla \left\vert \mathrm{Ric}%
\right\vert ^{2}\right\rangle \left\vert \mathrm{Ric}\right\vert
^{p-2}f^{a-1}\phi ^{q}  \notag \\
&&+\frac{q}{r}\int_{M}\left\langle \nabla f,\nabla \left\vert \mathrm{Ric}%
\right\vert ^{2}\right\rangle \left\vert \mathrm{Ric}\right\vert
^{p-2}f^{a}\phi ^{q-1}  \notag \\
&&-\int_{M}\left\langle \nabla f,\nabla \left\vert \mathrm{Ric}\right\vert
^{2}\right\rangle \left\vert \mathrm{Ric}\right\vert ^{p-2}f^{a}\phi ^{q} 
\notag \\
&&+c\int_{M}\left\vert \mathrm{Rm}\right\vert \left\vert \mathrm{Ric}%
\right\vert ^{p}f^{a}\phi ^{q}.  \notag
\end{eqnarray}
Note that 

\begin{eqnarray*}
-\int_{M}\left\langle \nabla f,\nabla \left\vert \mathrm{Ric}\right\vert
^{2}\right\rangle \left\vert \mathrm{Ric}\right\vert ^{p-2}f^{a}\phi ^{q}
&=&-\frac{2}{p}\int_{M}\left\langle \nabla f,\nabla \left\vert \mathrm{Ric}%
\right\vert ^{p}\right\rangle f^{a}\phi ^{q} \\
&=&\frac{2}{p}\int_{M}\left\vert \mathrm{Ric}\right\vert ^{p}\left( \Delta
f+a\left\vert \nabla f\right\vert ^{2}f^{-1}\right) f^{a}\phi ^{q} \\
&&-\frac{2q}{pr}\int_{M}\left\vert \mathrm{Ric}\right\vert ^{p}f^{a}\phi
^{q-1} \\
&\leq &c\int_{M}\left\vert \mathrm{Ric}\right\vert ^{p}f^{a}\phi ^{q}.
\end{eqnarray*}%
A similar argument also implies 

\begin{eqnarray*}
\frac{q}{r}\int_{M}\left\langle \nabla f,\nabla \left\vert \mathrm{Ric}%
\right\vert ^{2}\right\rangle \left\vert \mathrm{Ric}\right\vert
^{p-2}f^{a}\phi ^{q-1} &\leq &\frac{cp}{r}\int_{M}\left\vert \mathrm{Ric}%
\right\vert ^{p}f^{a}\phi ^{q-2}+C \\
&\leq &c\int_{M}\left\vert \mathrm{Ric}\right\vert ^{p}f^{a}\phi ^{q}+\frac{%
c\left( p\right) }{r^{\frac{q}{2}}}\int_{D\left( r\right) }\vert \mathrm{Ric}\vert^p\,f^{a}+C \\
&\leq &c\int_{M}\left\vert \mathrm{Ric}\right\vert ^{p}f^{a}\phi ^{q}+C
\end{eqnarray*}%
in view of (\ref{V}) and $\vert \mathrm{Ric}\vert \leq C.$

It is then easy to see that these estimates and (\ref{a11}) imply that 
\begin{eqnarray}
c\left( p\right) \int_{M}\left\vert \nabla \mathrm{Ric}\right\vert
^{2}\left\vert \mathrm{Ric}\right\vert ^{p-2}f^{a}\phi ^{q} &\leq &c\left(
p\right) \int_{M}\left\vert \mathrm{Ric}\right\vert ^{p}f^{a}\phi ^{q}
\label{a12} \\
&&+c\left( p\right) \int_{M}\left\vert \mathrm{Rm}\right\vert \left\vert 
\mathrm{Ric}\right\vert ^{p}f^{a}\phi ^{q}+C  \notag \\
&\leq &\frac{1}{p^{5}}\int_{M}\left\vert \mathrm{Rm}\right\vert
^{p}f^{a}\phi ^{q}  \notag \\
&&+c\left( p\right) \int_{M}\left\vert \mathrm{Ric}\right\vert ^{p}f^{a}\phi
^{q}+C.  \notag
\end{eqnarray}%
Using (\ref{a12}) and (\ref{a10}) we get that 

\begin{eqnarray}
&&\int_{M}\left\vert \nabla \mathrm{Ric}\right\vert \left\vert \nabla 
\mathrm{Rm}\right\vert \left\vert \mathrm{Ric}\right\vert \left\vert \mathrm{%
Rm}\right\vert ^{p-2-\alpha }f^{a}\phi ^{q}  \label{a13} \\
&\leq &\frac{1}{p^{5}}\int_{M}\left\vert \nabla \mathrm{Rm}\right\vert
^{2}\left\vert \mathrm{Rm}\right\vert ^{p-3+\alpha }f^{a}\phi ^{q}  \notag \\
&&+\frac{1}{p^{2}}\int_{M}\left\vert \nabla \mathrm{Ric}\right\vert
^{2}\left\vert \mathrm{Rm}\right\vert ^{p-1-\alpha }f^{a}\phi ^{q}  \notag \\
&&+\frac{1}{p^{5}}\int_{M}\left\vert \mathrm{Rm}\right\vert ^{p}f^{a}\phi
^{q}+c\left( p\right) \int_{M}\left\vert \mathrm{Ric}\right\vert
^{p}f^{a}\phi ^{q}+C.  \notag
\end{eqnarray}

We estimate the first term in (\ref{a6}) as follows. First, observe
that integration by parts yields

\begin{eqnarray}
&&\int_{M}\Delta _{f}\left( \left\vert \mathrm{Ric}\right\vert
^{2}\left\vert \mathrm{Rm}\right\vert ^{p-1-\alpha }\right) f^{a}\phi ^{q}
\label{a14} \\
&=&\int_{M}\Delta \left( \left\vert \mathrm{Ric}\right\vert ^{2}\left\vert 
\mathrm{Rm}\right\vert ^{p-1-\alpha }\right) f^{a}\phi ^{q}  \notag \\
&&-\int_{M}\left\langle \nabla f,\nabla \left(\left\vert \mathrm{Ric}\right\vert
^{2}\left\vert \mathrm{Rm}\right\vert ^{p-1-\alpha }\right)\right\rangle f^{a}\phi
^{q}  \notag \\
&=&\int_{M}\left\vert \mathrm{Ric}\right\vert ^{2}\left\vert \mathrm{Rm}%
\right\vert ^{p-1-\alpha }\Delta \left( f^{a}\phi ^{q}\right)   \notag \\
&&+\int_{M}\left\vert \mathrm{Ric}\right\vert ^{2}\left\vert \mathrm{Rm}%
\right\vert ^{p-1-\alpha }\left( \Delta f+a\left\vert \nabla f\right\vert
^{2}f^{-1}\right) f^{a}\phi ^{q}  \notag \\
&&-\frac{q}{r}\int_{M}\left\vert \mathrm{Ric}\right\vert ^{2}\left\vert 
\mathrm{Rm}\right\vert ^{p-1-\alpha }\left\vert \nabla f\right\vert
^{2}f^{a}\phi ^{q-1}.  \notag
\end{eqnarray}%
A direct computation shows that 

\begin{eqnarray*}
\Delta \left( f^{a}\phi ^{q}\right)  &\leq &\left( \Delta f^{a}\right) \phi
^{q}+f^{a}\Delta \phi ^{q} \\
&\leq &cp^{2}f^{a-1}\phi ^{q-2}.
\end{eqnarray*}%
Consequently, we get that  

\begin{eqnarray*}
\int_{M}\left\vert \mathrm{Ric}\right\vert ^{2}\left\vert \mathrm{Rm}%
\right\vert ^{p-1-\alpha }\Delta \left( f^{a}\phi ^{q}\right)  
&\leq& c\,p^{2}\int_{M}\left\vert \mathrm{Ric}\right\vert ^{2}\left\vert \mathrm{Rm}%
\right\vert ^{p-1-\alpha }f^{a-1}\phi ^{q-2} \\
&\leq &c\,p^{2}\int_{M}\left\vert \mathrm{Rm}\right\vert ^{p-1}f^{a-1}\phi
^{q-2}+C \\
&\leq &\frac{c}{p^{4}}\int_{M}\left\vert \mathrm{Rm}\right\vert
^{p}f^{a}\phi ^{q}+c\left( p\right) \int_{M}f^{a-p}\phi ^{q-2p} +C\\
&\leq &\frac{c}{p^{4}}\int_{M}\left\vert \mathrm{Rm}\right\vert
^{p}f^{a}\phi ^{q}+C,
\end{eqnarray*}%
where in the last line we have used (\ref{V}) and (\ref{const}) to
infer that $\int_{M}f^{a-p}\phi ^{q-2p}\leq C.$ 

Therefore, it follows from (\ref{a14}) that 
\begin{eqnarray}
&&\int_{M}\Delta _{f}\left( \left\vert \mathrm{Ric}\right\vert
^{2}\left\vert \mathrm{Rm}\right\vert ^{p-1-\alpha }\right) f^{a}\phi ^{q}
\label{a15} \\
&\leq &cp\int_{M}\left\vert \mathrm{Ric}\right\vert ^{2}\left\vert \mathrm{Rm%
}\right\vert ^{p-1-\alpha }f^{a}\phi ^{q}  \notag \\
&&+\frac{c}{p^{4}}\int_{M}\left\vert \mathrm{Rm}\right\vert ^{p}f^{a}\phi
^{q}+C  \notag \\
&\leq &\frac{c}{p^{4}}\int_{M}\left\vert \mathrm{Rm}\right\vert
^{p}f^{a}\phi ^{q}+c\left( p\right) \int_{M}\left\vert \mathrm{Ric}%
\right\vert ^{p}f^{a}\phi ^{q}+C,  \notag
\end{eqnarray}%
where we have used Young's inequality to obtain the inequality. 

Plugging (\ref{a13}) and (\ref{a15}) into (\ref{a6}), we conclude that

\begin{eqnarray}
cp^{3}\int_{M}\left\vert \nabla \mathrm{Ric}\right\vert ^{2}\left\vert 
\mathrm{Rm}\right\vert ^{p-1-\alpha }f^{a}\phi ^{q} 
&\leq &\frac{c}{p}
\int_{M}\left\vert \nabla \mathrm{Rm}\right\vert ^{2}\left\vert \mathrm{Rm}%
\right\vert ^{p-3+\alpha }f^{a}\phi ^{q}  \label{a19} \\
&&+\frac{c}{p}\int_{M}\left\vert \mathrm{Rm}\right\vert ^{p}f^{a}\phi ^{q} 
\notag \\
&&+c\left( p\right) \int_{M}\left\vert \mathrm{Ric}\right\vert ^{p}f^{a}\phi
^{q}+C  \notag
\end{eqnarray}
as 

\begin{eqnarray*}
cp \int_{M}\left\vert \mathrm{Ric}\right\vert ^{2}\left\vert \mathrm{Rm}%
\right\vert ^{p-\alpha }\,f^{a}\,\phi ^{q} \leq
\frac{1}{p}\int_{M}\left\vert \mathrm{Rm}\right\vert ^{p}\,f^{a}\,\phi ^{q} 
+c\left( p\right) \int_{M}\left\vert \mathrm{Ric}\right\vert ^{p}\,f^{a}\,\phi
^{q}+C.
\end{eqnarray*}

For the third term on the right hand side of (\ref{a5}), we now claim that
 
\begin{eqnarray}
&&\frac{cp}{r}\int_{M}\left\vert R_{ijkh}f_{h}\right\vert ^{2}\left\vert 
\mathrm{Rm}\right\vert ^{p-2}f^{a}\phi ^{q-1}  \label{a20} \\
&\leq &\frac{c}{p}\int_{M}\left\vert \nabla \mathrm{Rm}\right\vert
^{2}\left\vert \mathrm{Rm}\right\vert ^{p-3+\alpha }f^{a}\phi ^{q}  \notag \\
&&+\frac{c}{p}\int_{M}\left\vert \mathrm{Rm}\right\vert ^{p}f^{a}\phi ^{q}+C.
\notag
\end{eqnarray}%
Indeed, by (\ref{id}),
 
\begin{eqnarray*}
&&\frac{cp}{r}\int_{M}\left\vert R_{ijkh}f_{h}\right\vert ^{2}\left\vert 
\mathrm{Rm}\right\vert ^{p-2}f^{a}\phi ^{q-1} \\
&=&\frac{c\,p}{r}\int_{M}\nabla _{j}R_{ik}\left( R_{ijkh}f_{h}\right)
\left\vert \mathrm{Rm}\right\vert ^{p-2}f^{a}\phi ^{q-1} \\
&=&-\frac{c\,p}{r}\int_{M}R_{ik}f_{h}\left( \nabla _{j}R_{ijkh}\right)
\left\vert \mathrm{Rm}\right\vert ^{p-2}f^{a}\phi ^{q-1} \\
&&-\frac{c\,p}{r}\int_{M}R_{ik}R_{ijkh}f_{h}\left( \nabla _{j}\left\vert 
\mathrm{Rm}\right\vert ^{p-2}\right) f^{a}\phi ^{q-1} \\
&&-\frac{c\,p}{r}\int_{M}R_{ik}f_{hj}R_{ijkh}\left\vert \mathrm{Rm}\right\vert
^{p-2}f^{a}\phi ^{q-1} \\
&&-\frac{c\,a\,p}{r}\int_{M}R_{ik}R_{ijkh}f_{h}f_{j}\left\vert \mathrm{Rm}%
\right\vert ^{p-2}f^{a-1}\phi ^{q-1} \\
&&+\frac{c\,p\,(q-1)}{r^{2}}\int_{M}R_{ik}R_{ijkh}f_{h}f_{j}\left\vert \mathrm{Rm}%
\right\vert ^{p-2}f^{a}\phi ^{q-2}.
\end{eqnarray*}%
The last three terms above can be bounded by 
\begin{eqnarray*}
&&-\frac{c\,p}{r}\int_{M}R_{ik}f_{hj}R_{ijkh}\left\vert \mathrm{Rm}\right\vert
^{p-2}f^{a}\phi ^{q-1} \\
&&-\frac{c\,a\,p}{r}\int_{M}R_{ik}R_{ijkh}f_{h}f_{j}\left\vert \mathrm{Rm}%
\right\vert ^{p-2}f^{a-1}\phi ^{q-1} \\
&&+\frac{c\,q\,p}{r^{2}}\int_{M}R_{ik}R_{ijkh}f_{h}f_{j}\left\vert \mathrm{Rm}%
\right\vert ^{p-2}f^{a}\phi ^{q-2} \\
&\leq &\frac{c}{r}\int_{M}\left\vert \mathrm{Rm}\right\vert ^{p-1}f^{a}\phi
^{q-2} \\
&\leq &\frac{c}{p}\int_{M}\left\vert \mathrm{Rm}\right\vert ^{p}f^{a}\phi
^{q}+C.
\end{eqnarray*}%
Furthermore, note that 

\begin{eqnarray*}
&&-\frac{cp}{r}\int_{M}R_{ik}R_{ijkh}f_{h}\left( \nabla _{j}\left\vert 
\mathrm{Rm}\right\vert ^{p-2}\right) f^{a}\phi ^{q-1} \\
&&-\frac{cp}{r}\int_{M}R_{ik}f_{h}\left( \nabla _{j}R_{ijkh}\right)
\left\vert \mathrm{Rm}\right\vert ^{p-2}f^{a}\phi ^{q-1}\\
&\leq &\frac{cp^{2}}{\sqrt{r}}\int_{M}\left\vert \mathrm{Ric}\right\vert
\left\vert \nabla \mathrm{Rm}\right\vert \left\vert \mathrm{Rm}\right\vert
^{p-2}f^{a}\phi ^{q-1} \\
&\leq &\frac{1}{p}\int_{M}\left\vert \nabla \mathrm{Rm}\right\vert
^{2}\left\vert \mathrm{Rm}\right\vert ^{p-3+\alpha }f^{a}\phi ^{q}+\frac{%
c\left( p\right) }{r}\int_{M}\left\vert \mathrm{Ric}\right\vert
^{2}\left\vert \mathrm{Rm}\right\vert ^{p-1-\alpha }f^{a}\phi ^{q-2}.
\end{eqnarray*}%
Since 

\begin{eqnarray*}
&& \frac{c\left( p\right) }{r}\int_{M}\left\vert \mathrm{Ric}\right\vert
^{2}\left\vert \mathrm{Rm}\right\vert ^{p-1-\alpha }\,f^{a}\,\phi ^{q-2} \\
&\leq & \frac{1}{p}\,\int_{M}\left\vert \mathrm{Rm}\right\vert ^{p}\,f^{a}\,\phi^{q}+\frac{c\left( p\right) }{r^{\frac{p}{1+\alpha}}}\,
\int_{M}\left\vert \mathrm{Ric}\right\vert^{\frac{2p}{1+\alpha}}\,f^{a}\,
\phi ^{q-\frac{2p}{1+\alpha}} \\
&\leq & \frac{1}{p}\,\int_{M}\left\vert \mathrm{Rm}\right\vert ^{p}\,f^{a}\,
\phi^{q}+C,
\end{eqnarray*}
we see that 

\begin{eqnarray*}
&&-\frac{cp}{r}\int_{M}R_{ik}R_{ijkh}f_{h}\left( \nabla _{j}\left\vert 
\mathrm{Rm}\right\vert ^{p-2}\right) f^{a}\phi ^{q-1} \\
&\leq &\frac{1}{p}\int_{M}\left\vert \nabla \mathrm{Rm}\right\vert
^{2}\left\vert \mathrm{Rm}\right\vert ^{p-3+\alpha }f^{a}\phi ^{q}+\frac{1}{p%
}\int_{M}\left\vert \mathrm{Rm}\right\vert ^{p}f^{a}\phi ^{q}+C.
\end{eqnarray*}%
These estimates imply that (\ref{a20}) is indeed true. Putting (\ref{a19}) and
(\ref{a20}) into (\ref{a5}) we conclude that

\begin{eqnarray}
\int_{M}\left\vert \mathrm{Rm}\right\vert ^{p}f^{a}\phi ^{q} &\leq &\frac{c}{%
p}\int_{M}\left\vert \nabla \mathrm{Rm}\right\vert ^{2}\left\vert \mathrm{Rm}%
\right\vert ^{p-3+\alpha }f^{a}\phi ^{q}  \label{a21} \\
&&+c\left( p\right) \int_{M}\left\vert \mathrm{Ric}\right\vert ^{p}f^{a}\phi
^{q}+C.  \notag
\end{eqnarray}

We now use that 
\begin{equation*}
2\left\vert \nabla \mathrm{Rm}\right\vert ^{2}\leq \Delta \left\vert \mathrm{%
Rm}\right\vert ^{2}-\left\langle \nabla f,\nabla \left\vert \mathrm{Rm}%
\right\vert ^{2}\right\rangle +c\left\vert \mathrm{Rm}\right\vert ^{3}
\end{equation*}%
to estimate 
\begin{eqnarray*}
2\int_{M}\left\vert \nabla \mathrm{Rm}\right\vert ^{2}\left\vert \mathrm{Rm}%
\right\vert ^{p-3+\alpha }f^{a}\phi ^{q} &\leq &\int_{M}\left( \Delta
\left\vert \mathrm{Rm}\right\vert ^{2}\right) \left\vert \mathrm{Rm}%
\right\vert ^{p-3+\alpha }f^{a}\phi ^{q} \\
&&-\int_{M}\left\langle \nabla f,\nabla \left\vert \mathrm{Rm}\right\vert
^{2}\right\rangle \left\vert \mathrm{Rm}\right\vert ^{p-3+\alpha }f^{a}\phi
^{q} \\
&&+c\int_{M}\left\vert \mathrm{Rm}\right\vert ^{p+\alpha }f^{a}\phi ^{q} \\
&\leq &-a\int_{M}\left\langle \nabla f,\nabla \left\vert \mathrm{Rm}%
\right\vert ^{2}\right\rangle \left\vert \mathrm{Rm}\right\vert ^{p-3+\alpha
}f^{a-1}\phi ^{q} \\
&&+\frac{q}{r}\int_{M}\left\langle \nabla f,\nabla \left\vert \mathrm{Rm}%
\right\vert ^{2}\right\rangle \left\vert \mathrm{Rm}\right\vert ^{p-3+\alpha
}f^{a}\phi ^{q-1} \\
&&-\int_{M}\left\langle \nabla f,\nabla \left\vert \mathrm{Rm}\right\vert
^{2}\right\rangle \left\vert \mathrm{Rm}\right\vert ^{p-3+\alpha }f^{a}\phi
^{q} \\
&&+c\int_{M}\left\vert \mathrm{Rm}\right\vert ^{p+\alpha }f^{a}\phi ^{q}.
\end{eqnarray*}%
However, 
\begin{eqnarray*}
&&-\int_{M}\left\langle \nabla f,\nabla \left\vert \mathrm{Rm}\right\vert
^{2}\right\rangle \left\vert \mathrm{Rm}\right\vert ^{p-3+\alpha }f^{a}\phi
^{q} \\
&=&-\frac{2}{p-1+\alpha }\int_{M}\left\langle \nabla f,\nabla \left\vert 
\mathrm{Rm}\right\vert ^{p-1+\alpha }\right\rangle f^{a}\phi ^{q} \\
&=&\frac{2}{p-1+\alpha }\int_{M}\left( \left( \Delta f\right)
f^{a}+a\left\vert \nabla f\right\vert ^{2}f^{a-1}\right) \left\vert \mathrm{%
Rm}\right\vert ^{p-1+\alpha }\phi ^{q} \\
&&-\frac{2q}{p-1+\alpha }\frac{1}{r}\int_{M}\left\vert \nabla f\right\vert
^{2}f^{a}\left\vert \mathrm{Rm}\right\vert ^{p-1+\alpha }\phi ^{q-1} \\
&\leq &c\int_{M}\left\vert \mathrm{Rm}\right\vert ^{p-1+\alpha }f^{a}\phi
^{q}.
\end{eqnarray*}%
Similarly, 
\begin{eqnarray*}
&&-a\int_{M}\left\langle \nabla f,\nabla \left\vert \mathrm{Rm}\right\vert
^{2}\right\rangle \left\vert \mathrm{Rm}\right\vert ^{p-3+\alpha
}f^{a-1}\phi ^{q} \\
&\leq &cp\int_{M}\left\vert \mathrm{Rm}\right\vert ^{p-1+\alpha }f^{a-1}\phi
^{q} \\
&\leq &c\int_{M}\left\vert \mathrm{Rm}\right\vert ^{p-1+\alpha }f^{a}\phi
^{q}+C.
\end{eqnarray*}%
Finally, a similar argument implies that 
\begin{eqnarray*}
&&\frac{q}{r}\int_{M}\left\langle \nabla f,\nabla \left\vert \mathrm{Rm}%
\right\vert ^{2}\right\rangle \left\vert \mathrm{Rm}\right\vert ^{p-3+\alpha
}f^{a}\phi ^{q-1} \\
&\leq &\frac{cp}{r}\int_{M}\left\vert \mathrm{Rm}\right\vert ^{p-1+\alpha
}f^{a}\phi ^{q-2}+C \\
&\leq &c\int_{M}\left\vert \mathrm{Rm}\right\vert ^{p+\alpha }f^{a}\phi ^{q}+%
\frac{c\left( p\right) }{r^{p+\alpha }}\int_{M}f^{a}\phi ^{q-2\left(
p+\alpha \right) }+C \\
&\leq &c\int_{M}\left\vert \mathrm{Rm}\right\vert ^{p+\alpha }f^{a}\phi
^{q}+C,
\end{eqnarray*}%
where we have used (\ref{const}) and (\ref{V}) in the last line. The above
estimates show that 
\begin{eqnarray*}
\int_{M}\left\vert \nabla \mathrm{Rm}\right\vert ^{2}\left\vert \mathrm{Rm}%
\right\vert ^{p-2}f^{a}\phi ^{q} &\leq &c\int_{M}\left\vert \mathrm{Rm}%
\right\vert ^{p-1+\alpha }f^{a}\phi ^{q} \\
&&+c\int_{M}\left\vert \mathrm{Rm}\right\vert ^{p+\alpha }f^{a}\phi ^{q}+C.
\end{eqnarray*}%
Plugging this in (\ref{a21}), we arrive at
 
\begin{eqnarray*}
\int_{M}\left\vert \mathrm{Rm}\right\vert ^{p}f^{a}\phi ^{q} &\leq &\frac{c}{%
p}\int_{M}\left\vert \mathrm{Rm}\right\vert ^{p-1+\alpha }f^{a}\phi ^{q}+%
\frac{c}{p}\int_{M}\left\vert \mathrm{Rm}\right\vert ^{p+\alpha }f^{a}\phi
^{q} \\
&&+c\left( p\right) \int_{M}\left\vert \mathrm{Ric}\right\vert ^{p}f^{a}\phi
^{q}+C.
\end{eqnarray*}%
This proves the lemma.
\end{proof}

Using Lemma \ref{int} we obtain the following crucial result. From now on,
we assume $p\geq 8n$, $a\leq \frac{1}{4}p$ and $q\geq 2p+5.$

\begin{proposition}
\label{Rm}Let $\left( M,g,f\right) $ be an $n$ dimensional shrinking Ricci
soliton with $\lim_{x\to \infty} \left\vert \mathrm{Ric}\right\vert (x)=0.$
Then 
\begin{equation*}
\int_{M}\left\vert \mathrm{Rm}\right\vert ^{p}f^{a}\phi ^{q}\leq c\left(
p\right) \int_{M}\left\vert \mathrm{Ric}\right\vert ^{p}f^{a}\phi ^{q}+C.
\end{equation*}
\end{proposition}

\begin{proof}
Applying Lemma \ref{int} for $\alpha =1$ we get%
\begin{eqnarray}
\int_{M}\left\vert \mathrm{Rm}\right\vert ^{p}f^{a}\phi ^{q} &\leq &\frac{c}{%
p}\int_{M}\left\vert \mathrm{Rm}\right\vert ^{p+1}f^{a}\phi ^{q}  \label{a22}
\\
&&+c\left( p\right) \int_{M}\left\vert \mathrm{Ric}\right\vert ^{p}f^{a}\phi
^{q}+C.  \notag
\end{eqnarray}%
Since $q\geq 2\left( p+1\right) +3,$ we may apply Lemma \ref{int} for $%
\alpha =0$ and conclude that%

\begin{eqnarray}
\int_{M}\left\vert \mathrm{Rm}\right\vert ^{p+1}f^{a}\phi ^{q} &\leq &\frac{c%
}{p}\int_{M}\left\vert \mathrm{Rm}\right\vert ^{p}f^{a}\phi ^{q}  \label{a23}
\\
&&+c\left( p\right) \int_{M}\left\vert \mathrm{Ric}\right\vert ^{p}f^{a}\phi
^{q}+C.  \notag
\end{eqnarray}%
Proposition \ref{Rm} follows by combining (\ref{a22}) and (\ref{a23}).
\end{proof}

This proposition enables us to obtain the following bound for Ricci
curvature. We continue to assume that $p\geq 8n$, $a\leq \frac{1}{4}p$
and $q\geq 2p+5$.

\begin{proposition}
\label{Ric}Let $\left( M,g,f\right) $ be an $n$ dimensional shrinking Ricci
soliton with $\lim_{x\to \infty} \left\vert \mathrm{Ric}\right\vert (x)=0.$
Then 
\begin{equation*}
\int_{M}\left\vert \mathrm{Ric}\right\vert ^{p}f^{a}<\infty .
\end{equation*}
\end{proposition}

\begin{proof}
Recall that 
\begin{equation*}
\Delta R_{ij}-\left\langle \nabla R_{ij},\nabla f\right\rangle
=R_{ij}-2R_{ikjl}R_{ij}R_{kl}.
\end{equation*}%
This implies that 
\begin{equation*}
\left\vert \mathrm{Ric}\right\vert ^{2}\leq \frac{1}{2}\Delta \left\vert 
\mathrm{Ric}\right\vert ^{2}-\frac{1}{2}\left\langle \nabla f,\nabla
\left\vert \mathrm{Ric}\right\vert ^{2}\right\rangle +2\left\vert \mathrm{Rm}%
\right\vert \left\vert \mathrm{Ric}\right\vert ^{2}-\left\vert \nabla 
\mathrm{Ric}\right\vert ^{2}.
\end{equation*}%
Therefore, 
\begin{eqnarray}
\int_{M}\left\vert \mathrm{Ric}\right\vert ^{p}f^{a}\phi ^{q} &=&\frac{1}{2}%
\int_{M}\Delta \left\vert \mathrm{Ric}\right\vert ^{2}\left\vert \mathrm{Ric}%
\right\vert ^{p-2}f^{a}\phi ^{q}  \label{d1} \\
&&-\frac{1}{2}\int_{M}\left\langle \nabla f,\nabla \left\vert \mathrm{Ric}%
\right\vert ^{2}\right\rangle \left\vert \mathrm{Ric}\right\vert
^{p-2}f^{a}\phi ^{q}  \notag \\
&&+2\int_{M}\left\vert \mathrm{Rm}\right\vert \left\vert \mathrm{Ric}%
\right\vert ^{p}f^{a}\phi ^{q}  \notag \\
&&-\int_{M}\left\vert \nabla \mathrm{Ric}\right\vert ^{2}\left\vert \mathrm{%
Ric}\right\vert ^{p-2}f^{a}\phi ^{q}.  \notag
\end{eqnarray}%
Integrating by parts, we get 
\begin{eqnarray*}
\frac{1}{2}\int_{M}\Delta \left\vert \mathrm{Ric}\right\vert ^{2}\left\vert 
\mathrm{Ric}\right\vert ^{p-2}f^{a}\phi ^{q} &\leq &-\frac{a}{2}%
\int_{M}\left\langle \nabla f,\nabla \left\vert \mathrm{Ric}\right\vert
^{2}\right\rangle \left\vert \mathrm{Ric}\right\vert ^{p-2}f^{a-1}\phi ^{q}
\\
&&+\frac{q}{2r}\int_{M}\left\langle \nabla f,\nabla \left\vert \mathrm{Ric}%
\right\vert ^{2}\right\rangle \left\vert \mathrm{Ric}\right\vert
^{p-2}f^{a}\phi ^{q-1} \\
&\leq &cp\int_{M}\left\vert \nabla \mathrm{Ric}\right\vert \left\vert 
\mathrm{Ric}\right\vert ^{p-1}f^{a-\frac{1}{2}}\phi ^{q-1} \\
&\leq &\int_{M}\left\vert \nabla \mathrm{Ric}\right\vert ^{2}\left\vert 
\mathrm{Ric}\right\vert ^{p-2}f^{a}\phi ^{q}+cp^{2}\int_{M}\left\vert 
\mathrm{Ric}\right\vert ^{p}f^{a-1}\phi ^{q-2}.
\end{eqnarray*}%
However, (\ref{a2}) and Young's inequality imply 
\begin{eqnarray*}
cp^{2}\int_{M}\left\vert \mathrm{Ric}\right\vert ^{p}f^{a-1}\phi ^{q-2}
&\leq &\frac{c}{p}\int_{M}\left\vert \mathrm{Ric}\right\vert
^{p-1}f^{a-1}\phi ^{q-2} +C\\
&\leq &\frac{1}{p}\int_{M}\left\vert \mathrm{Ric}\right\vert ^{p}f^{a}\phi
^{q}+C.
\end{eqnarray*}%
This proves that 
\begin{eqnarray}
\frac{1}{2}\int_{M}\Delta \left\vert \mathrm{Ric}\right\vert ^{2}\left\vert 
\mathrm{Ric}\right\vert ^{p-2}f^{a}\phi ^{q} &\leq &\int_{M}\left\vert
\nabla \mathrm{Ric}\right\vert ^{2}\left\vert \mathrm{Ric}\right\vert
^{p-2}f^{a}\phi ^{q}  \label{d2} \\
&&+\frac{1}{p}\int_{M}\left\vert \mathrm{Ric}\right\vert ^{p}f^{a}\phi
^{q}+C.  \notag
\end{eqnarray}%
According to Proposition \ref{Rm}, there exists a constant $c_{0}\left(
p\right) >0$ so that 
\begin{equation*}
\int_{M}\left\vert \mathrm{Rm}\right\vert ^{p}f^{a}\phi ^{q}\leq c_{0}\left(
p\right) \int_{M}\left\vert \mathrm{Ric}\right\vert ^{p}f^{a}\phi ^{q}+C.
\end{equation*}%
For this constant $c_{0}\left( p\right) ,$ we use Young's inequality to
conclude%
\begin{eqnarray}
\int_{M}\left\vert \mathrm{Rm}\right\vert \left\vert \mathrm{Ric}\right\vert
^{p}f^{a}\phi ^{q} &\leq &\frac{1}{p}\frac{1}{c_{0}\left( p\right) }%
\int_{M}\left\vert \mathrm{Rm}\right\vert ^{p}f^{a}\phi ^{q}  \label{d3} \\
&&+c_{1}\left( p\right) \int_{M}\left\vert \mathrm{Ric}\right\vert ^{\frac{%
p^{2}}{p-1}}f^{a}\phi ^{q}  \notag \\
&\leq &\frac{1}{p}\int_{M}\left\vert \mathrm{Ric}\right\vert ^{p}f^{a}\phi
^{q}  \notag \\
&&+c_{1}\left( p\right) \int_{M}\left\vert \mathrm{Ric}\right\vert ^{p+\frac{p}{p-1}}
f^{a}\phi ^{q}+C.  \notag
\end{eqnarray}%
Since $\lim_{x\to \infty} \left\vert \mathrm{Ric}\right\vert=0,$ there
exists $r_{1}>0$ so that on $M\backslash D\left( r_{1}\right) $%
\begin{equation*}
\left\vert \mathrm{Ric}\right\vert ^{\frac{p}{p-1}}\leq \frac{1}{%
pc_{1}\left( p\right) }.
\end{equation*}%
Hence, (\ref{d3}) yields that 
\begin{equation}
\int_{M}\left\vert \mathrm{Rm}\right\vert \left\vert \mathrm{Ric}\right\vert
^{p}f^{a}\phi ^{q}\leq \frac{2}{p}\int_{M}\left\vert \mathrm{Ric}\right\vert
^{p}f^{a}\phi ^{q}+C,  \label{d4}
\end{equation}%
where $C$ depends on the geometry of $D\left( r_{1}\right) $.

Using (\ref{d2}) and (\ref{d4}) in (\ref{d1}) implies that 
\begin{equation}
\int_{M}\left\vert \mathrm{Ric}\right\vert ^{p}f^{a}\phi ^{q}\leq
-\int_{M}\left\langle \nabla f,\nabla \left\vert \mathrm{Ric}\right\vert
^{2}\right\rangle \left\vert \mathrm{Ric}\right\vert ^{p-2}f^{a}\phi ^{q}+C.
\label{d5}
\end{equation}%
However, as $a\leq \frac{p}{4}$ and $p\geq 8n$, we get 
\begin{eqnarray*}
-\int_{M}\left\langle \nabla f,\nabla \left\vert \mathrm{Ric}\right\vert
^{2}\right\rangle \left\vert \mathrm{Ric}\right\vert ^{p-2}f^{a}\phi ^{q}
&=&-\frac{2}{p}\int_{M}\left\langle \nabla f,\nabla \left\vert \mathrm{Ric}%
\right\vert ^{p}\right\rangle f^{a}\phi ^{q} \\
&\leq &\frac{2}{p}\int_{M}\left\vert \mathrm{Ric}\right\vert ^{p}\left(
\left( \Delta f\right) f^{a}+a\left\vert \nabla f\right\vert
^{2}f^{a-1}\right) \phi ^{q} \\
&\leq &\frac{2}{3}\int_{M}\left\vert \mathrm{Ric}\right\vert ^{p}f^{a}\phi
^{q}.
\end{eqnarray*}%
Together with (\ref{d5}), we obtain

\begin{equation*}
\int_{M}\left\vert \mathrm{Ric}\right\vert ^{p}f^{a}\phi ^{q}\leq C.
\end{equation*}
\end{proof}

Combining Proposition \ref{Rm} with Proposition \ref{Ric}, one concludes
that 
\begin{equation}
\int_{M}\left\vert \mathrm{Rm}\right\vert ^{p}f^{a}\leq C.  \label{f1}
\end{equation}%
We are now ready to prove the main theorem of the paper.

\begin{theorem}
\label{decay_1}Let $\left( M,g,f\right) $ be a gradient shrinking Ricci
soliton of dimension $n$ with Ricci curvature convergent to zero at
infinity. Then $\left( M,g,f\right) $ is $C^k$ asymptotic to a cone for all $%
k.$
\end{theorem}

\begin{proof}
Applying (\ref{f1}) for $a=\frac{p}{4}$ implies that 
\begin{equation*}
\int_{B_{x}\left( 1\right) }\left\vert \mathrm{Rm}\right\vert ^{p}\leq
C\left( d\left( x_{0},x\right) +1\right) ^{-\frac{p}{2}}.
\end{equation*}%
Note that 
\begin{equation*}
\Delta \left\vert \mathrm{Rm}\right\vert ^{2}\geq -u\left\vert \mathrm{Rm}%
\right\vert ^{2},
\end{equation*}%
where $u:=c\left( \left\vert \mathrm{Rm}\right\vert +f\right) .$
Furthermore, the Sobolev constant of $B_{x}\left( 1\right)$ depends only on dimension, 
the Ricci curvature bound
and Perelman's invariant \cite{MWa}. So the Moser iteration \cite{L} implies that 
\begin{eqnarray*}
\left\vert \mathrm{Rm}\right\vert \left( x\right) &\leq &C\left(
\int_{B_{x}\left( 1\right) }u^{n}+1\right) ^{\frac{1}{p}}\left(
\int_{B_{x}\left( 1\right) }\left\vert \mathrm{Rm}\right\vert ^{p}\right) ^{%
\frac{1}{p}} \\
&\leq &C\left( d\left( x_{0},x\right) +1\right) ^{-\frac{1}{4}}.
\end{eqnarray*}%
In particular, this shows that $\lim_{x\to \infty}\left\vert \mathrm{Rm}\right\vert=0.$ 
Now from 
\begin{equation*}
\Delta _{f}\left\vert \mathrm{Rm}\right\vert \geq \left\vert \mathrm{Rm}%
\right\vert -c\left\vert \mathrm{Rm}\right\vert ^{2}
\end{equation*}%
and the information that $\left\vert \mathrm{Rm}\right\vert \rightarrow 0$
at infinity, it follows that $\left\vert \mathrm{Rm}\right\vert $ decays
quadratically, that is,
 
\begin{equation}
\left\vert \mathrm{Rm}\right\vert \left( x\right) \leq c\left( d\left(
x_{0},x\right) +1\right) ^{-2}.  \label{f2}
\end{equation}%
Indeed, it was shown in \cite{MW} that if a nonnegative function $w$
satisfies 
\begin{equation*}
\Delta _{f}w\geq w-cw^{2}
\end{equation*}%
on a shrinking Ricci soliton and $w\rightarrow 0$ at infinity, then there
exists a constant $c>0$ so that $w\leq \frac{c}{f}.$ The proof given in \cite%
{MW} is in dimension $n=4,$ but it is easy to see that the same argument
works in any dimension.

Now (\ref{f2}) and Shi's derivative estimates imply the derivatives of the
Riemann curvature tensor satisfy%
\begin{equation*}
|\nabla ^{k}\mathrm{Rm}|\left( x\right) \leq c\,\left( d\left(
x_{0},x\right) +1\right) ^{-k-2}
\end{equation*}%
for all $k\geq 1.$ From this, it follows that $\left( M,g\right) $ is $C^{k}$ asymptotic to
a cone for all $k.$ We refer to \cite{KW} for more details. The theorem is
proved.
\end{proof}


\begin{thebibliography}{99}
\bibitem{BW} C. B\"{o}hm and B. Wilking, Manifolds with positive curvature
operators are space forms, Ann. Math. (2) 167 (2008), no. 3, 1079-1097.

\bibitem{Ca} H. D. Cao, Existence of gradient K\"{a}hler-Ricci solitons,
Elliptic and parabolic methods in geometry (B. Chow, R. Gulliver, S. Levy,
J. Sullivan, editors), AK Peters, 1996, 1-16.

\bibitem{Ca1} H.D. Cao, Recent progress on Ricci solitons, Adv. Lect. Math.
11, no. 2, 1-38, Int. Press, Somerville, MA, 2010.

\bibitem{Ca2} H.D. Cao, Geometry of complete gradient shrinking Ricci
solitons, Adv. Lect. Math., 17, no. 1, 227-246, Int. Press, Somerville, MA,
2011.

\bibitem{CCZ} H. D. Cao, B. L. Chen, X. P. Zhu, Recent developments on
Hamilton's Ricci flow. Surveys in differential geometry. Vol. XII. Geometric
flows, 47-112, Surv. Differ. Geom., 12, Int. Press, Somerville, MA, 2008.

\bibitem{CC} H. D. Cao and Q. Chen, On Bach-flat gradient shrinking Ricci solitons,
Duke Math. J. 162 (2013), no.6, 1149-1169.

\bibitem{CZ} H.D. Cao and D. Zhou, On complete gradient shrinking Ricci
solitons, J. Differential Geom. 85 (2010), no. 2, 175-186.

\bibitem{CWZ} X. Cao, B. Wang and Z. Zhang, On locally conformally flat gradient shrinking Ricci solitons,
Commun. Contemp. Math. 13 (2011), no. 2, 269-282.

\bibitem{C} B.L. Chen, Strong uniqueness of the Ricci flow, J. Differential
Geom. 82 (2009), no. 2, 362-382.

\bibitem{CLN} B. Chow, P. Lu and L. Ni, Hamilton's Ricci flow. Graduate
Studies in Mathematics, 77. American Mathematical Society, Providence, RI;
Science Press, New York, 2006.

\bibitem{DW} A. Dancer and M. Wang, On Ricci solitons of cohomogeneity one,
Ann. Global Anal. Geom. 39 (2011), no. 3, 259-292.

\bibitem{ELM} M. Eminenti, G. La Nave and C. Mantegazza, Ricci solitons: the equation point of view,
Manuscripta Math. 127 (2008), 345-367.

\bibitem{EMT} J. Enders, R. M\"uller, P. Topping, On Type-I singularities in
Ricci flow, Comm. Anal. Geom, 19 (2011), no. 5, 905-922.

\bibitem{FIK} M. Feldman, T. Ilmanen and D. Knopf, Rotationally symmetric
shrinking and expanding gradient K\"{a}hler-Ricci solitons, J. Differential
Geom. 65 (2003), no. 2, 169-209.

\bibitem{H} R. Hamilton, The formation of singularities in the Ricci flow,
Surveys in Differential Geom. 2 (1995), 7-136, International Press.

\bibitem{K} N. Koiso, On rotationally symmetric Hamilton's equation for 
K\"{a}hler-Einstein metrics, Advanced Studies in Pure Math., 18 (1), 327-337,
Kinokuniya (Tokyo) and Academic Press (Boston), 1990.

\bibitem{KW} B. Kotschwar and L. Wang, Rigidity of Asymptotically Conical
Shrinking Gradient Ricci Solitons, J. Differential Geom. (to appear).

\bibitem{L} P. Li, Geometric Analysis, Cambridge Studies in Advanced
Mathematics, 2012.

\bibitem{MS} O. Munteanu and N. Sesum, On gradient Ricci solitons, 
J. Geom. Anal. 23 (2013), no. 2, 539-561.

\bibitem{MW} O. Munteanu and J. Wang, Geometry of shrinking Ricci solitons,
arXiv:1410.3813.

\bibitem{MWa} O. Munteanu and M.T. Wang, The curvature of gradient Ricci
solitons, Math. Res. Lett. 18 (2011), no. 6, 1051-1069.

\bibitem{N} A. Naber, Noncompact Shrinking 4-Solitons with Nonnegative
Curvature, J. Reine Angew. Math., 645 (2010), 125-153.

\bibitem{NW} L. Ni and N. Wallach, On a classification of gradient shrinking
solitons, Math. Res. Lett. 15 (2008), no. 5, 941-955.

\bibitem{P} G. Perelman, The entropy formula for the Ricci flow and its
geometric applications, arXiv:math. DG/0211159.

\bibitem{PW} P. Petersen and P. Wylie, On the classification of gradient Ricci solitons, 
Geom. Topol. 14 (2010), 2277-2300.

\bibitem{S} W. X. Shi, Deforming the metric on complete Riemannian
manifolds, J. Differential Geom. 30 (1989), no. 1, 223--301.

\bibitem{WZ} X.J. Wang and X. Zhu, K\"{a}hler-Ricci solitons on toric
manifolds with positive first Chern class, Adv. Math. 188 (2004), 87-103.

\bibitem{Z} Z. H. Zhang, Gradient shrinking solitons with vanishing Weyl tensor, 
Pacific J. Math. 242 (2009), 189-200.

\end{thebibliography}
\end{document}